\documentclass{article}

\usepackage{amssymb}
\usepackage{amsmath,amsfonts,amsthm,amstext}
\usepackage{dsfont}
\usepackage[english]{babel}
\usepackage[utf8]{inputenc}
\usepackage[all]{xy}
\usepackage[pdftex]{hyperref}
\usepackage[pdftex]{color,graphicx}
\usepackage{mathrsfs}
\usepackage{appendix}
\usepackage{enumerate}
\usepackage{todonotes}
\usepackage{csquotes}
\usepackage{textcomp}
\usepackage{yfonts}
\usepackage{mathtools}

\theoremstyle{plain}
\newtheorem{theorem}{Theorem}

\newtheorem{lemma}[theorem]{Lemma}

\title{Waring–Goldbach problems for one square and higher powers}

\author{
  Geovane Matheus Lemes Andrade \\
  \texttt{gmla.andrade@gmail.com}
}

\date{}

\begin{document}
\maketitle

\begin{center}
Department of Mathematics, University of Brasília, Brasília, DF, Brazil
\end{center}

\renewcommand{\thefootnote}{}
\renewcommand{\thefootnote}{\arabic{footnote}}

\begin{abstract}
    We prove that every sufficiently large odd integer can be expressed as a sum of one square and fourteen fifth powers, all of primes. In addition, we establish that every sufficiently large even integer can be written as a sum of one square, one biquadrate, and twelve fifth powers of primes.
\end{abstract}

\section{Introduction}

The problem of representing large integers as sums of one square together with higher powers of natural numbers has a long history. The earliest significant contributions are due to Stanley \cite{Stanley1930}, Theorems 10–12. She demonstrated that the minimal number $s(k)$ required to guarantee that every sufficiently large integer can be expressed as a sum of one square and $s(k)$ $k$th powers satisfies
$$s(3) \leq 7, \hspace{0.3cm} s(4) \leq 14, \hspace{0.3cm} s(5) \leq 28, \hspace{0.3cm} s(k) \ll 2^{k-2}\left(\tfrac{1}{2}k-1\right), \hspace{0.3cm} (k>5).$$

Over time, these bounds have been refined. The strongest currently known results are $s(3) \leq 5$, due to Watson \cite{Watson1972}; Br\"udern and Wooley \cite{partiBrudernWooley2024} proved that $s(4)=7$, and moreover
$$s(k) \leq 2k-1 \hspace{0.4cm}(3\leq k \leq 6), \qquad s(k) \leq 2k \hspace{0.4cm}(7\leq k \leq 11),$$
while for all $k\geq 3$ one has
$$s(k) \leq \left\lfloor\left(\tfrac{3}{4}+2\log 2\right)k\right\rfloor+2.$$
Furthermore, they showed that every sufficiently large integer can be written as a sum of one square of prime and $s_1(k)$ $k$th powers of natural numbers, where
$$s_1(3) \leq 6, \qquad s_1(4) \leq 8, \qquad s_1(5) \leq 12.$$

When all variables are restricted to primes, only a few results are known. For $k=3$, the best approximation to date is due to Li and Zhang \cite{LiZhang2018}, who showed that every sufficiently large even integer can be expressed as a sum of a square of a natural number with at most six prime factors, together with five cubes of primes. For $k=4$, L\"u and Cai \cite{LuCai2019} established representations involving one square and nine biquadrates of primes. When $k=5$, M. Zhang, J. Li, and F. Xue \cite{zhang2024arXiv240203076Z} established a Waring–Goldbach analogue of Br\"udern and Kawada \cite{BrudernKawada2011}, proving that every sufficiently large even integer can be expressed as a sum of one square and seventeen fifth powers of primes. In this paper, we improve this result by proving the following theorem.

\begin{theorem}\label{teo2}
    For each sufficiently large odd integer $n$, the equation
    \begin{equation}\label{eq2}
        n=x_1^2+x_2^5+\dots+x_{15}^5
    \end{equation}
    admits a solution in which all variables are prime. Moreover, for every sufficiently large even integer $n$, the equation
    \begin{equation}\label{eq3}
        n=x_1^2+x_2^4+x_3^5+\dots+x_{14}^5
    \end{equation}
    has a solution with all $x_j$ prime.
\end{theorem}

These results are obtained via the circle method. We combine mean value estimates due to Hooley \cite{Hooley1981}, Br\"udern \cite{Bru87}, Thanigasalam \cite{Thanigasalam1989,Thanigasalam1994}, and Kawada and Wooley \cite{KawadaWooley2001}. After applying H\"older’s inequality, a fractional contribution involving a fifth power remains to be estimated. This part is handled using recent advances in the Vinogradov mean value theorem \cite{KumWoo16}, together with Kumchev’s bounds \cite{Kum06} for exponential sums over primes. A pruning argument then completes the proof.

\textit{Notation}. Throughout this paper, $p$ denotes a prime number, $e(\alpha)$ abbreviates $\exp(2\pi i \alpha)$, and $\phi(q)$ is Euler’s totient function. The implicit constants in Vinogradov’s and Landau’s symbols depend on the value assigned to $\epsilon$.

\section{Preliminaries}

Let $n$ be a large positive integer. For $k \geq 2$, define
$$P_k = n^{1/k}, \qquad L = \log n, \qquad f_k(\alpha)=\sum \limits_{\frac{1}{2}P_k < p \leq P_k} e(\alpha p^k)\log p.$$

We now introduce parameters extracted from \cite{KawadaWooley2001}. Set 
\begin{align*}
    \lambda_j&=\left(\frac{33}{40}\right)^{j-1} \qquad (1\leq j \leq 6),\\
    \\
    \lambda_7&=\left(\frac{33}{40}\right)^5\frac{136}{163},\qquad
    \lambda_8=\left(\frac{33}{40}\right)^5\frac{576}{815},\qquad \lambda_9 = \left(\frac{33}{40}\right)^5\frac{512}{815},\\ 
    \\
    \Lambda&=2\lambda_9+\sum_{j=1}^{8}\lambda_j.
\end{align*}
A straightforward computation yields $\Lambda = 5(1-\lambda)$, with 
\begin{align}\label{mu}
    \lambda=\frac{304\,729\,213}{83\,456\,000\,000}.
\end{align}

For $1\leq j \leq 9$, let 
\begin{align*}
    g_{j}(\alpha) &= \sum \limits_{\frac{1}{2}P_5^{\lambda_j} < p \leq P_5^{\lambda_j}} e(\alpha p^5)\log p.
\end{align*}

Define 
\begin{align*}
    \mathcal{G}(\alpha) = g_9(\alpha)^2\prod_{j=1}^{8}g_j(\alpha).
\end{align*}

For brevity, we often write $f_k$ and $g_j$ for $f_k(\alpha)$ and $g_j(\alpha)$, respectively. We also introduce
\begin{align*}
    F_1(\alpha)&=f_2(\alpha)f_5(\alpha)^4\mathcal{G}(\alpha),\\
    F_2(\alpha) &= f_2(\alpha)f_4(\alpha)f_5(\alpha)^2\mathcal{G}(\alpha).
\end{align*}

For $j=1,2$, by orthogonality, the integrals
\begin{align}\label{ort}
    \nu_j(n) = \int_{0}^{1} F_j(\alpha)e(-\alpha n)\,d\alpha
\end{align}
count the number of solutions of \eqref{eq2} and \eqref{eq3}, respectively, in prime variables restricted to certain intervals and weighted by a product of $\log x_j$. Define 
\begin{align*}
    \Theta_1 &= \frac{3}{10}+\frac{\Lambda}{5},\\
    \Theta_2 &= \frac{3}{20}+\frac{\Lambda}{5}.
\end{align*}

Our goal is to establish the lower bound
$$\nu_j(n) \gg n^{\Theta_j}$$
for all large $n\equiv j \pmod 2$.

 Fix a real number $B \geq 1$. Let $\mathfrak{M}$ be the union of all intervals 
\begin{align}\label{mqamajor}
    \mathfrak{M}(q,a) = \{\alpha \in[0,1]: |\alpha - a/q| \leq L^B/n \},
\end{align} 
with $0\leq a \leq q$, $(a,q)=1$ and $1\leq q \leq L^B$. Let $\mathfrak{m} = [0,1]\setminus \mathfrak{M}$. We aim to show that for all sufficiently large $n\equiv j \pmod 2$, 
\begin{align}\label{task}
    \int \limits_{\mathfrak{M}} F_j(\alpha)e(-\alpha n)d\alpha \gg n^{\Theta_j}, \qquad \int \limits_{\mathfrak{m}} F_j(\alpha)e(-\alpha n)d\alpha \ll n^{\Theta_j}L^{-1}.
\end{align}
With this, Theorem \ref{teo2} will be established. We begin by analyzing the contribution from $\mathfrak{M}$.

\section{The major arcs}

For $k \geq 2$ and $0<\eta\leq 1$, define 
\begin{align*}
    S_k(q,a) = \sum \limits_{\substack{x=1 \\ (x,q)=1}}^{q} e\left(\frac{ax^k}{q} \right), \qquad v_{k,\eta}(\beta) = \frac{1}{k}\sum \limits_{2^{-k}n^{\eta}< m \leq n^{\eta}} m^{1/k - 1} e(\beta m).
\end{align*}

Following Hua’s proof of Lemma 6 in \cite{Hua38}, we establish the result below.

\begin{lemma}\label{hua6}
    Let $a,q$ be integers with $0\leq a \leq q$, $(a,q)=1$, and $1\leq q \leq L^B$. Then, for $\alpha \in \mathfrak{M}(q,a)$, one has 
    $$ \sum \limits_{\frac{1}{2}P_k^{\eta} < p \leq P_k^{\eta}} e(\alpha p^k)\log p = \frac{S_k(q,a)}{\phi(q)}v_{k,\eta}(\alpha -a/q) + O(P_k^{\eta}e^{-c_1\sqrt{L}}),$$
    where $c_1>0$ is an absolute constant.
\end{lemma}

We will apply Lemma \ref{hua6} with $\eta=\lambda_r$. For convenience, we write $v_{5,r}$ for $v_{5,\lambda_r}$. When $\eta=1$, we simply write $v_k$ instead of $v_{k,1}$.

\begin{lemma}\label{major}
    Let $1\leq j\leq 2$. For large $n\equiv j \pmod 2$ one has 
    $$ \int \limits_{\mathfrak{M}} F_j(\alpha)e(-\alpha n)d\alpha \gg n^{\Theta_j}.$$
\end{lemma}

\begin{proof}
    We begin by approximating $F_j(\alpha)$ using the auxiliary functions $S_k(q,a)$ and $v_{k,\ell}$. Suppose $\alpha \in \mathfrak{M}(q,a)$ with $q \leq L^B$ and $(a,q)=1$. Let $\beta = \alpha - a/q$, and define 
    \begin{align*}
        U_1(q,a)&=S_2(q,a)S_5(q,a)^{14},\\
        w_1(\beta)&=v_{2}(\beta)v_{5}(\beta)^4v_{5,9}(\beta)^2\prod_{j=1}^{8}v_{5,j}(\beta),\\
        \\
        U_2(q,a)&=S_2(q,a)S_4(q,a)S_5(q,a)^{12},\\
        w_2(\beta)&=v_{2}(\beta)v_{4}(\beta)v_{5}(\beta)^2v_{5,9}(\beta)^2\prod_{j=1}^{8}v_{5,j}(\beta).
    \end{align*}
   
    Let $\ell_1=15$ and $\ell_2=14$. By Lemma \ref{hua6}, we obtain 
    $$F_j(\alpha) =  \phi(q)^{-\ell_j}U_j(q,a)w_j(\alpha - a/q) + O(n^{1+\Theta_j}L^{-4B}).$$
    Integrating this over $\mathfrak{M}$, a set of measure $O(L^{3B}/n)$, yields
    \begin{align}
         \int \limits_{\mathfrak{M}}F_j&(\alpha)e(-\alpha n) d\alpha \nonumber \\&= \sum \limits_{q \leq L^B}\frac{A_{n,j}(q)}{\phi(q)^{\ell_j}} \int_{-L^B/n}^{L^B/n} w_j(\beta)e(-\beta n)d\beta + O(n^{\Theta_j} L^{-B}), \label{over}
     \end{align} 
     where 
     $$A_{n,j}(q) = \sum \limits_{\substack{a=1 \\ (a,q)=1}}^{q} U_j(q,a)e(-an/q).$$
     
    By Lemma 5 of \cite{Hua38}, we have $U_j(q,a) \ll q^{\ell_j/2 + \epsilon}$ for coprime $a,q$, and hence $A_{n,j}(q) \ll q^{\frac{\ell_j+2}{2}+\epsilon}$ uniformly in $n$. It follows that the series 
    $$ \mathfrak{S}_j(n) = \sum \limits_{q =1}^{\infty}\frac{A_{n,j}(q)}{\phi(q)^{\ell_j}}$$ 
    converges absolutely. Moreover, routine calculations show that 
    \begin{align*}
        \sum \limits_{q \leq L^B}\frac{A_{n,j}(q)}{\phi(q)^{\ell_j}} = \mathfrak{S}_j(n) + O(L^{-B})
    \end{align*}
    holds uniformly in $n$. Also, by Lemma 6.2 of \cite{bookVau}, one has
    \begin{align}\label{wcestimate}
        w_j(\beta) \ll n^{\Theta_j +1}(1+n|\beta|)^{-2}.
    \end{align}
    Thus, the sum in \eqref{over} can be replaced by $\mathfrak{S}_j(n)$ with negligible error. 

    We aim to prove that $ \mathfrak{S}_j(n) \gg 1$ for all large $n\equiv j \pmod 2$. Following the argument used in the proof of Lemma 2.11 of \cite{bookVau}, one can likewise establish that $A_{n,j}(q)$ is multiplicative. Furthermore, Lemma 4 of \cite{Hua38} shows that $A_{n,j}(q) = 0$ unless $q$ is square-free. This yields
    \begin{align}\label{prod}
        \mathfrak{S}_j(n) = \prod \limits_{p}(1+(p-1)^{-\ell_j}A_{n,j}(p)).
    \end{align}

    For $j=1,2$, let $M_{n,j}(p)$ denote the number of solutions of \eqref{eq2} and \eqref{eq3}, respectively, in the finite field $\mathbb{F}_p$, with all variables non-zero. By orthogonality modulo $p$, 
$$1 + (p-1)^{-\ell_j} A_{n,j}(p)= p(p-1)^{-\ell_j}M_{n,j}(p).$$
Applying the Cauchy–Davenport theorem (Lemma 2.14 of \cite{bookVau}) gives $M_{n,j}(p) \geq 1$, except when  $p=2$ and $2 \mid n$ for $j=1$, and $p=2$ with $2 \nmid n$ for $j=2$. Hence the factors in \eqref{prod} are positive for all $n\equiv j \pmod 2$, and the uniform convergence with respect to $n$ yields a number $p_0$ such that, for $n\equiv j \pmod 2$,
$$\mathfrak{S}_j(n) \geq \frac{1}{2} \prod \limits_{p\leq p_0}(1+(p-1)^{-\ell_j}A_{n,j}(p)) \geq \frac{1}{2} \prod \limits_{p\leq p_0}p^{-\ell_j+1}.$$
Thus, the lower bound $ \mathfrak{S}_j(n) \gg 1$ holds for all large $n\equiv j \pmod 2$.

To conclude the proof, it remains to show that the integral on the right-hand side of \eqref{over} is $\gg n^{\Theta_j}$. Using the estimate \eqref{wcestimate}, we deduce 
\begin{align}\label{right}
     \int_{-L^B/n}^{L^B/n} w_j(\beta)e(-\beta n)d\beta = \int_{-1/2}^{1/2} w_j(\beta)e(-\beta n)d\beta + O(n^{\Theta_j}L^{-B}).
\end{align}

Suppose $j=1$. By orthogonality, the integral on the right-hand side of \eqref{right} is 
\begin{align}\label{sum2}
    \frac{1}{2\cdot 5^{14}} \sum (m_1)^{-\frac{1}{2}}\left(\prod_{j=2}^{15}m_j\right)^{-\frac{4}{5}},
\end{align}
where the sum runs over $m_1,\dots,m_{15}$ subject to 
\begin{align}\label{sixsum2}
   \sum_{j=1}^{15} m_j=n,
\end{align}
with
 \begin{align}
    & 2^{-2}n<m_1 \leq n, \qquad 2^{-5}n<m_j \leq n \qquad (j=2,3,4,5), \label{restr11}\\
    &2^{-5}n^{\lambda_{j-5}}<m_j\leq n^{\lambda_{j-5}} \qquad (j=6,\dots,15). \label{restr22}
\end{align}
For any choice of $m_1 \leq \tfrac{17}{64}n$, $m_2,m_3,m_4 \leq \tfrac{3}{64}n$, and  $m_j \leq \tfrac{3}{64}n^{\lambda_{j-5}}$ for $j=6,\dots,15$, in accordance with \eqref{restr11} and \eqref{restr22}, one can solve \eqref{sixsum2} with $m_5$ satisfying \eqref{restr11}. Hence, the sum in \eqref{sum2} admits the lower bound $n^{\Theta_1}$.

Suppose $j=2$. The integral on the right-hand side of \eqref{right} is
\begin{align}\label{sum3}
    \frac{1}{2\cdot 4 \cdot 5^{12}} \sum (m_1)^{-\frac{1}{2}}(m_2)^{-\frac{3}{4}}\left(\prod_{j=3}^{14}m_j\right)^{-\frac{4}{5}},
\end{align}
where the sum extends over $m_1,\dots,m_{14}$ subject to 
\begin{align}\label{sixsum3}
   \sum_{j=1}^{14} m_j=n,
\end{align}
with
 \begin{align}
    & 2^{-2}n<m_1 \leq n, \qquad 2^{-4}n<m_2 \leq n, \label{restr111}\\
    &2^{-5}n<m_3,m_4\leq n, \qquad2^{-5}n^{\lambda_{j-4}}<m_j\leq n^{\lambda_{j-4}} \qquad (j=5,\dots,12), \label{restr222} \\
    &2^{-5}n^{\lambda_9}<m_{13},m_{14}\leq n^{\lambda_9}. \label{restr33}
\end{align}
Whenever $m_1 \leq \tfrac{17}{64}n$, $m_2\leq \tfrac{5}{64}n$, $m_3 \leq \tfrac{3}{64}n$,  $m_j \leq \tfrac{3}{64}n^{\lambda_{j-4}}$ for $j=5,\dots,12$, and $m_{13},m_{14} \leq \tfrac{3}{64}n^{\lambda_{9}}$ in accordance with \eqref{restr111}, \eqref{restr222}, and \eqref{restr33}, one can solve \eqref{sixsum3} with $m_4$ satisfying \eqref{restr222}. Consequently, the sum in \eqref{sum3} is bounded below by $n^{\Theta_2}$. This completes the proof of the lemma.

\end{proof}

\section{The minor arcs}

A preliminary remark is in order. The mean values arising in the proofs of the subsequent lemma can be bounded above by a power of $L$ multiplied by the number of integer solutions to an associated Diophantine equation. This perspective allows us to dispense with the logarithmic weights and thereby regard $f_k$ as classical Weyl sums. 

\begin{lemma}\label{estimatehooley}
    Let $k_1,k_2,k_3\geq 3$ be integers satisfying $\tfrac{1}{k_1}+\tfrac{1}{k_2}+\tfrac{1}{k_3}\geq \tfrac{3}{5}$. Then  
    \begin{align*}
        \int_{0}^{1}|f_2f_{k_1}f_{k_2}f_{k_3}|^2 \, d\alpha \ll n^{2\left( \frac{1}{k_1}+\frac{1}{k_2}+\frac{1}{k_3}\right)+\epsilon}.
    \end{align*}
\end{lemma}

\begin{proof}
    By Lemma 11 of \cite{Hooley1981}, the number $N_0$ of solutions to 
    $$z_1^2-z_2^2=x_1^{k_1}+x_2^{k_2}+x_3^{k_3}- y_1^{k_1}-y_2^{k_2}-y_3^{k_3},$$
    with $ \tfrac{1}{2}P_{2}<z_1,z_2\leq P_{2}$ and $\tfrac{1}{2}P_{k_j}<x_j,y_j\leq P_{k_j} $ for $j=1,2,3$, where the right-hand side equals zero, satisfies 
    $$N_0\ll P_2n^{\frac{7}{6}\left( \frac{1}{k_1}+\frac{1}{k_2}+\frac{1}{k_3}\right)+\epsilon}.$$ 
    The number of choices for $x_j,y_j$ that make the right-hand side nonzero is $O(P_{k_1}^2P_{k_2}^2P_{k_3}^2)$. For each such choice, both $x_1-x_2$ and $x_1+x_2$ divide this side. Thus, by a divisor-counting argument, there are $O(P_2^\epsilon)$ possibilities for $x_1,x_2$. This gives 
    \begin{align*}
        \int_{0}^{1}|f_2f_{k_1}f_{k_2}f_{k_3}|^2 \, d\alpha &\ll P_2n^{\frac{7}{6}\left( \frac{1}{k_1}+\frac{1}{k_2}+\frac{1}{k_3}\right)+\epsilon} + P_2^{\epsilon}n^{2\left( \frac{1}{k_1}+\frac{1}{k_2}+\frac{1}{k_3}\right)} \\
        &\ll n^{2\left( \frac{1}{k_1}+\frac{1}{k_2}+\frac{1}{k_3}\right)+\epsilon},
    \end{align*}
    where, in the final estimate, we use the hypothesis $\tfrac{1}{k_1}+\tfrac{1}{k_2}+\tfrac{1}{k_3}\geq \tfrac{3}{5}$. 
\end{proof}

Consider the following mean values:
\begin{align*}
    J_1 &= \int_{0}^{1} |f_{2}|^2|f_{5}|^6 \, d\alpha, \qquad
    J_2 = \int_{0}^{1} |f_{2}|^2|f_4|^4 \, d\alpha,\qquad
    J_3=\int_{0}^{1} |\mathcal{G}(\alpha)|^2 \, d\alpha.
\end{align*}
Applying Lemma \ref{estimatehooley} with $k_1=k_2=k_3=5$, together with Lemma 1 of \cite{Bru87}, we obtain
\begin{align}\label{j1j2}
    J_1 \ll n^{\frac{6}{5}+\epsilon}, \qquad J_2 \ll n^{1+\epsilon}.
\end{align}
Moreover, Lemma 6.1 of Kawada \& Wooley \cite{KawadaWooley2001} yields 
\begin{align}\label{j3j4}
    J_3 &\ll n^{1-\lambda+\epsilon},
\end{align}
where $\lambda$ is defined in \eqref{mu}.

With this preparation, we isolate the contribution to \eqref{ort} arising from the sets
\begin{align*}
    \mathcal{E}_1 &= \{ \alpha \in [0,1] : |f_{5}(\alpha)| \leq P_5^{1-3\lambda-10\tau} \},\\
     \mathcal{E}_2 &= \{ \alpha \in [0,1] : |f_{5}(\alpha)| \leq P_5^{1-5\lambda-20\tau} \},
\end{align*}
where $\tau = 2^{-100}$.

\begin{lemma}\label{epsilon}
   For $1\leq j \leq 2$, one has 
   $$ \int \limits_{\mathcal{E}_j} F_j(\alpha)e(-\alpha n)d\alpha \ll n^{\Theta_j - \tau}.$$
\end{lemma}

\begin{proof}
    By H\"older’s inequality, we deduce that 
    \begin{align*}
        \int \limits_{\mathcal{E}_1} |F_1(\alpha)|\,d\alpha &\leq J_1^{1/2}J_3^{1/2} \sup \limits_{\alpha \in \mathcal{E}_1}|f_{5}(\alpha)|,\\
        \int \limits_{\mathcal{E}_2} |F_2(\alpha)|\,d\alpha &\leq J_1^{1/4}J_2^{1/4}J_3^{1/2} \sup \limits_{\alpha \in \mathcal{E}_2}|f_{5}(\alpha)|^{1/2}.
    \end{align*}
    Employing the bounds in \eqref{j1j2} and \eqref{j3j4}, together with the definition of $\mathcal{E}_j$ and a straightforward calculation, the desired estimate follows.
\end{proof}

For the complement of $\mathcal{E}_j$, we apply Lemma 2.2 from \cite{KumWoo16} in conjunction with partial summation, taking $k=5$ and $X = \tfrac{1}{2}P_5$. It follows that 
$$[0,1]\setminus \mathcal{E}_j \subset \mathfrak{R},$$
where $\mathfrak{R}$ is the union of the disjoint intervals
$$ \mathfrak{R}(q,a) = \{\alpha \in [0,1]: |q\alpha -a| \leq n^{-\frac{118}{119}}\},$$
with $0\leq a \leq q$, $(a,q)=1$, and $1 \leq q \leq n^{\frac{1}{119}}$. 

Define a function $\Upsilon:\mathfrak{R} \to [0,1]$ by 
$$\Upsilon(\alpha) = (q+n|q\alpha - a|)^{-1} \qquad (\alpha \in \mathfrak{R}(q,a)).$$
Then, by Theorem 2 of \cite{Kum06} and partial summation, there exists a constant $C>0$ such that for $\alpha \in \mathfrak{R}$ and $k=2,4,5$, one has 
\begin{align*}
    f_{k}(\alpha) &\ll P_k L^C\Upsilon(\alpha)^{1/2 - \epsilon}.
\end{align*}
Observe that $g_1(\alpha)=f_5(\alpha)$.
Thus, using the trivial bound for $g_r(\alpha)$, we obtain
\begin{align*}
    F_j(\alpha) &\ll n^{1+\Theta_j}L^{5C}\Upsilon(\alpha)^{39/16}.
\end{align*}

We now integrate over $\mathfrak{R}\setminus \mathfrak{M}$. Routine calculations yield 
$$ \int \limits_{\mathfrak{R}\setminus \mathfrak{M}} |F_j(\alpha)|d\alpha \ll n^{1+\Theta_j}L^{5C}\int \limits_{\mathfrak{R}\setminus \mathfrak{M}} |\Upsilon(\alpha)|^{39/16}d\alpha \ll  n^{\Theta_j}L^{5C-\frac{7}{16}B}.$$
This establishes the following result.

\begin{lemma}\label{ult}
    Let $B\geq 80C + 16$. Then 
    $$\int \limits_{\mathfrak{R}\setminus \mathfrak{M}} |F_j(\alpha)|d\alpha \ll  n^{\Theta_j}L^{-1}.$$
\end{lemma}

We can now complete the proof of Theorem \ref{teo2}. Choose $B$ so that Lemma \ref{ult} applies. By the definitions of the sets $\mathfrak{M}, \mathfrak{R}$, and $\mathcal{E}_j$, one has 
$$\mathfrak{m}\subset \mathcal{E}_j \cup \mathfrak{R}\setminus \mathfrak{M}.$$
Therefore, combining Lemmas \ref{epsilon}, \ref{ult}, and \ref{major}, we obtain \eqref{task}, which completes the proof of Theorem \ref{teo2}.

\section*{Funding}

This work was supported by CAPES (Coordenação de Aperfeiçoamento de Pessoal de Nível Superior, Brazil).

\section*{Acknowledgements}

The author gratefully acknowledges the Graduate Program in Mathematics at the University of Brasília for the academic support provided during the development of this research. 

\section*{Declaration of Interest}

The author declares that there are no known competing financial interests or personal relationships that could have influenced the work reported in this paper.

\bibliographystyle{elsarticle-num}
\bibliography{sample}

\end{document}